\documentclass[11pt]{article}

\usepackage{amsmath, amsthm, amssymb, amscd}
\usepackage[title]{appendix}

\usepackage{hyperref}

\numberwithin{equation}{section}
\newtheorem{thm}{Theorem}[section]

\newcommand{\Var}{\operatorname{Var}}
\newtheorem{lem}{Lemma}[section]
\newtheorem{rem}{Remark}[section]

\newcommand{\Cov}{\operatorname{Cov}}

\newcommand{\E}{\operatorname{E}}

\newcommand{\MA}{\operatorname{MA}}
\newcommand{\Tr}{\operatorname{tr}}

\begin{document}

\title{Lower bounds for volatility estimation in microstructure noise models} 

\author{Axel Munk$^{1,2}$ and Johannes Schmidt-Hieber$^1$ \\
\vspace{0.1cm} \\
{\em Institut f\"ur Mathematische Stochastik, Universit\"at G\"ottingen,} \\ {\em Goldschmidtstr. 7, 37077 G\"ottingen} \\ {\small {\em Email:} \texttt{munk@math.uni-goettingen.de}, \texttt{schmidth@math.uni-goettingen.de}}}

\footnotetext[1]{The research of Axel Munk and Johannes Schmidt-Hieber was supported by DFG Grant FOR 916 and GK 1023.}
\footnotetext[2]{Author for correspondence, email:  \texttt{munk@math.uni-goettingen.de}}

\date{}
\maketitle

%
%
%
%



\begin{abstract}
In this paper we derive lower bounds in minimax sense for estimation of the instantaneous volatility if the diffusion type part cannot be observed directly but under some additional Gaussian noise. Three different models are considered. Our technique is based on a general inequality for Kullback-Leibler divergence of multivariate normal random variables and spectral analysis of the processes. The derived lower bounds are indeed optimal. Upper bounds can be found in \cite{mun}. Our major finding is that the Gaussian microstructure noise introduces an additional degree of ill-posedness for each model, respectively.
\end{abstract}

\medskip



%
%
%
%
%
%

\section{Introduction and Discussion}
Let $X_t=\int_0^t \sigma\left(s\right) dW_s,$ where $\left(W_t\right)_{t\in \left[0,1\right]}$ denotes here and in the following a standard Brownian motion. Consider the model
\begin{eqnarray}
	Y_{i,n}&= X_{i/n}
	+\tau \epsilon_{i,n}, 
	\quad i=1,\ldots,n,
	\label{eq.mod2}
\end{eqnarray}
where $\epsilon_{i,n} \sim \mathcal{N}\left(0,1\right),$ i.i.d. Throughout the paper $\left(W_t\right)_{t\in \left[0,1\right]}$ and $\left(\epsilon_{1,n},\ldots,\epsilon_{n,n}\right)$ are assumed to be independent. Further $\sigma$ is an unknown, positive and deterministic function and $\tau>0$ is a known constant.

In the financial econometrics literature variations of model (\ref{eq.mod2}) are often denoted as high-frequency models, since $\left(W_t\right)_{t\in \left[0,1\right]}$ is sampled on time points $t=i/n$ corresponding to short time intervals $1/n$ which can be of the magnitude of seconds, nowadays. There is a vast amount of literature on volatility estimation in high-frequency models with additional microstructure noise term (see Barndorff-Nielsen et {\it al.} \cite{barn}, Jacod et {\it al.} \cite{jac}, Zhang \cite{zha2} and Zhang et {\it al.} \cite{zha1} among many others). These kinds of models have attained a lot of attention recently, since the usual quadratic variation techniques for estimation of $\int_0^1\sigma^2(s)ds$ lead to inconsistent estimators (cf. Zhang \cite{zha2}). Brown et {\it al.} \cite{bro03} studied low- and high-frequency volatility models by means of asymptotic equivalence. Recently, Rei\ss \ \cite{rei2} has shown that model (\ref{eq.mod2}) is asymptotically equivalent to a Gaussian shift experiment.

\noindent
Closely related to model (\ref{eq.mod2}) is
\begin{eqnarray}
	\tilde Y_{i,n}&= \sigma\left(\frac in\right)W_{i/n}
	+\tau \epsilon_{i,n}, 
	\quad i=1,\ldots,n.
	\label{eq.mod}
\end{eqnarray}
This model can be regarded as a nonparametric extension of the model with constant $\sigma, \tau$ as discussed by Gloter and Jacod \cite{glo1}, \cite{glo2} and for variogram estimation by Stein \cite{ste}. As a further natural modification of (\ref{eq.mod2}), we consider
\begin{eqnarray}
	\bar Y_{i,n}= \int_0^{i/n} X_s ds
	+\tau \epsilon_{i,n}, 
	\quad i=1,\ldots,n.
	\label{eq.mod3}
\end{eqnarray}
For constant $\sigma$, the process $$\left(\int_0^{t} X_s ds\right)_{t\geq 0} \stackrel{\mathcal{D}}{=} \left(\int_0^{t} \left(t-s\right)\sigma\left(s\right) dW_s\right)_{t\geq 0}$$ is called integrated Brownian. In this case, model (\ref{eq.mod3}) has been used as a prior model for nonparametric regression (see e.g. Cox \cite{cox}).

\noindent
All models have a common structure: They might be interpreted as observations coming from a particular Volterra type stochastic integral, i.e. $\int_0^t K\left(s,t\right)\sigma\left(s,t\right) dW_s$ under additional measurement noise. In this paper we derive lower bounds for the models (\ref{eq.mod2})-(\ref{eq.mod3}). We stress that the treatment for general $K$ is an interesting but rather difficult task.


\noindent
More precisely, we derive minimax lower bounds for estimation of the instantaneous volatility, i.e. $\sigma^2$ as a function of time, with respect to $L_2$-loss. One of the key steps is to consider the estimation problem after taking finite differences which is typical for variance estimation (see Brown and Levine \cite{bro07} or Munk et {\it al.} \cite{mun2}). Usually in nonparametric regression, lower bounds are obtained under an independence assumption on the observations. In order to deal with dependend data, we introduce a new bound of the Kullback-Leibler divergence which might be of interest by its own. The lower bounds then follow from a standard multiple testing argument together with this bound and the control of the eigenvalues of the covariance operator.

\noindent
In nonparametric variance estimation, $n^{-\alpha/\left(2\alpha+1\right)}$ is well-known to be the minimax rate of convergence given H\"older-smoothness $\alpha$ (for a definition see (\ref{eq.hoelddef})) of the variance and a sufficiently smooth regression function (see Brown and Levine \cite{bro07} or Munk and Ruymgaart \cite{mun3}). We show that for the microstructure noise models (\ref{eq.mod2}) and (\ref{eq.mod}) the lower bounds are $n^{- \alpha/\left(4\alpha+2\right)}$ and for model (\ref{eq.mod3}), $n^{- \alpha/\left(8\alpha+4\right)}.$ If $\sigma$ exceeds some minimal required smoothness these rates are shown to be also upper bounds for models (\ref{eq.mod2}) and (\ref{eq.mod}) (Munk and Schmidt-Hieber \cite{mun}). 

\noindent
For constant $\sigma,$ i.e. $Y_{i,n}=\sigma W_{i/n}+\tau\epsilon_{i,n},$ $8\tau \sigma^3n^{-1/4}$ is the optimal asymptotic variance for estimation of $\sigma^2.$ First this has been shown in a more general setting by Gloter and Jacod \cite{glo1} but can be also derived  as in Cai et {\it al.} \cite{cai} with respect to minimax risk via the information inequality method (see Brown and Farrell \cite{bro90} and Brown and Low \cite{bro91}). In order to explain the optimal rate $n^{-1/4}$ it is tempting to think that this is related to the pathwise smoothness of the Brownian motion, since optimal estimation of a H\"older continuous function with index $1/2$ results in a $n^{-1/4}$ rate of convergence. However, this reasoning is in general not true since the smoothness of integrated Brownian motion is arbitrarily close to $3/2.$ By the same argument we should obtain in model (\ref{eq.mod3}) an $n^{-3/8}$ rate if $\sigma$ constant. This contradicts the obtained lower bound $n^{-1/8}.$ 

\noindent
Indeed, we find it more properly to look at these models from the viewpoint of statistical inverse problems. The eigenvalues of the Brownian motion and integrated Brownian motion covariance operators behave like $1/i^2$ and $1/i^4,$ respectively (see Freedman \cite{fre}). So the maximal frequency $i$ for estimation of $\sigma$ is reached when $1/i^2\sim 1/n$ (for Brownian motion) and $1/i^4\sim 1/n$ (for integrated Brownian motion), i.e. in the first case we may use $O\left(n^{1/2}\right)$ frequencies and in the second case $O\left(n^{1/4}\right)$ resulting in the reduction of the rate of convergence by a factor of $1/2$ and $1/4,$ respectively. Motivated by this heuristics, we conjecture that the optimal rate of convergence for the kernel $K\left(s,t\right)=\left(t-s\right)^q, q\in [0,\infty)$ and $\sigma\left(s,t\right)=\sigma\left(s\right)$ is $n^{-\alpha/\left[\left(2q+2\right)\left(2\alpha+1\right)\right]}.$ This implies that the rate of convergence decreases as the order of the zero of $K$ increases, or equivalently, as smoother the path of $\int_0^t K\left(s,t\right)\sigma\left(s\right)dW_s$ becomes.
For a further result in this direction ($q\in(0,1/2),$ $\sigma$ constant) see Gloter and Hoffmann \cite{glo4}. Note that for $q\in [1/2,1)\cup (1,\infty)$ the spectral decomposition of the integrated Brownian motion is not known and hence our strategy of proof cannot be applied. More general techniques are required.

\noindent
Note finally that the lower bounds still hold if we consider generalizations of the models. For instance consider model (\ref{eq.mod2}) and allow $\sigma$ to be random itself. Assume further that $\epsilon=\left(\epsilon_{1,n},\ldots,\epsilon_{n,n}\right)$ is general microstructure noise, i.e. white noise with bounded moments (see Huang et {\it al.} \cite{hua} or Zhou \cite{zho}). In this model the lower bound for estimation of the instantaneous volatility is still $n^{-\alpha/\left(4\alpha+2\right)},$ of course. To show that this is indeed the upper bound is current research by Marc Hoffmann and the authors, where the estimator is constructed by wavelet techniques and methods as in Brown et {\it al.} \cite{bro08} are employed.

\section{Results}
The next Lemma might be of interest by its own and can be applied to various problems in variance estimation where we do not have independent observations. The Lemma can be viewed as a generalization of Lemma 2.1 in Golubev et {\it al.} \cite{gol} for matrices with non-uniformly bounded eigenvalues. See also inequality (3.8) of Rei\ss \ \cite{rei}. For our purpose it is required to allow eigenvalue sequences tending to $0$ and $\infty$. The proof of our minimax results follows standard arguments (see Tsybakov \cite{tsyb}). Note, however, that the underlying dependency structure of the process causes severe technical difficulties and requires very sharp bounds for the Kullback-Leibler distance in these models as given in the next lemma. Recall that the Kullback-Leibler divergence between two probability measures $P,Q$ is defined as
\begin{align*}
	d_K\left(P,Q\right):=\int \log\left(\frac{dP}{dQ}\right) dP,
\end{align*}
whenever $P\ll Q$ and $+\infty$ otherwise.

\begin{lem}
\label{lem.kllem}
Let $X\sim\mathcal{N}\left(\mu, \Sigma_0\right)$ and $Y\sim\mathcal{N}\left(\mu, \Sigma_1\right)$ be $n$-variate normal r.v's with expectation $\mu$ and covariance $\Sigma_0$ and $\Sigma_1$, respectively and denote by $P_{X}$ and $P_Y$ the corresponding probability measures. Assume $0< C\Sigma_0\leq \Sigma_1$ for some constant $0<C\leq 1$. Then it holds for the Kullback-Leibler divergence
\begin{align}
	d_K(P_Y,P_X)
		\leq
		\frac 1{4C^2} \left\|\Sigma_0^{-1/2}\left(\Sigma_1-\Sigma_0\right)\Sigma_0^{-1/2}\right\|_F^2
		\leq
		\frac 1{4C^2}
		\left\|\Sigma_0^{-1}\Sigma_1-I_n\right\|_F^2,
	\label{eq.kullineq}
\end{align}
where $I_n$ denotes the $n \times n$ dimensional identity matrix and $\left\|\cdot \right\|_F$ is the Frobenius norm, i.e. for a matrix $A$, $\left\|A\right\|_F:=\Tr^{1/2}\left(AA^t\right)$.
\end{lem}

\begin{proof}
Note that
\begin{eqnarray*}
	d_K(P_Y,P_X)
		=
		\frac 12\left(
		\log\left(\det\left(\Sigma_0^{-1/2}\Sigma_1\Sigma_0^{-1/2}\right)^{-1}\right)
		+\Tr\left(\Sigma_{0}^{-1/2}\Sigma_1\Sigma_{0}^{-1/2}\right)-n\right).
\end{eqnarray*}
Introduce $\Sigma:=\Sigma_0^{-1/2}\Sigma_1 \Sigma_0^{-1/2}$ and note that $\Sigma$ is positive definite. Further, by assumption $0< C\Sigma_0\leq \Sigma_1$ and hence $\Sigma_0^{-1/2}\Sigma_1\Sigma_0^{-1/2}\geq CI_n$. This gives $w_i:=\lambda_i\left(\Sigma\right)\geq C, \quad i=1,\ldots,n.$ Recall $\det\left(\Sigma\right)= \prod_{i=1}^n w_i$ and $\Tr\left(\Sigma\right)=\sum_{i=1}^n w_i$. Hence
\begin{align*}
	d_K(P_Y,P_X)
	&=\frac 12
	\left(
	-\sum_{i=1}^n \log\left(w_i\right)+\sum_{i=1}^n w_i -n
	\right)
	=
	\frac 12
	\sum_{i=1}^n
	\left(
	w_i-\log\left(1+w_i-1\right)-1
	\right)
	.
\end{align*}
Assume $x\geq C-1$. Expand $-\log\left(1+x\right)=-x+\left(2\left(1+\xi\right)^{2}\right)^{-1}x^2$ for a suitable $\left|\xi\right|\leq \left|x\right|$. For $C-1\leq x \leq 0$ we have $-\log\left(1+x\right)\leq-x+x^2/\left(2C^2\right)$ and for $x\geq 0$, $-\log\left(1+x\right)\leq-x+x^2/2$. Therefore by Lemma \ref{lem.froblem}
\begin{eqnarray*}
	d_K(P_Y,P_X)
	\leq \frac 1{4C^2} \sum_{i=1}^n \left(w_i-1\right)^2
	=\frac 1{4C^2}
	\left\|\Sigma_0^{-1/2}\left(\Sigma_1-\Sigma_0\right)\Sigma_0^{-1/2}\right\|_F^2.
\end{eqnarray*}
\end{proof}

\begin{rem}
The assumption $0<C\Sigma_0\leq \Sigma_1$ can be relaxed at the cost of an additional symmetrization term in (\ref{eq.kullineq}). More precisely, assume in Lemma \ref{lem.kllem} instead of $0<C\Sigma_0\leq \Sigma_1$ that $0<\Sigma_0, \Sigma_1$ holds. Then we have 
\begin{eqnarray*}
	d_K(P_Y,P_X)
		\leq
		\frac 1{4}
		\left\|\Sigma_0^{-1}\Sigma_1-I_n\right\|_F^2
		+\frac 1{4}
		\left\|\Sigma_1^{-1}\Sigma_0-I_n\right\|_F^2		
		.
\end{eqnarray*}
\end{rem}

Next we prove a lower bound for H\"older continuous functions $\sigma^2$ and $\tau^2$. For this we need some notation. We write $\left[x\right]:=\max_{z \in \mathbb{Z}}\left\{z\leq x\right\}$, $x\in \mathbb{R}$, the integer part of $x$. Let $0<l<u<\infty$ some constants. The class of uniformly bounded H\"older continuous  functions of index $\alpha$ on the interval $I$ is defined by
\begin{align}
	&\mathcal{C}^b\left(\alpha, L\right)
	:= \mathcal{C}^b\left(\alpha, L, \left[l,u\right]\right) := 
	\left\{
	f: f^{\left(p\right)} \ \text{exists for } p=\left[\alpha\right],
	\right. \label{eq.hoelddef} \\  & \quad \left.
	\quad \left|f^{\left(p\right)}(x)-f^{\left(p\right)}(y)\right|
	\leq L\left|x-y\right|^{\alpha-p}, \ \forall x,y \in I, \ 0<l\leq f\leq u <\infty
	\right\}. \notag
\end{align}
For any function $g$ we introduce the forward difference operator $\Delta_i g := g\left(\left(i+1\right)/n\right)-g\left(i/n\right).$ $\log()$ is defined to be the binary logarithm and we write $\mathbb{M}_p$ and $\mathbb{D}_p$ for the space of $p \times p$ matrices and $p \times p$ diagonal matrices over $\mathbb{R}$, respectively. The $i$-th largest eigenvalue of a Hermitian matrix $M$ is defined as $\lambda_i\left(M\right).$

\begin{thm}
\label{thm.lbl2}
Assume model (\ref{eq.mod2}) and $\alpha>1/2$ or model (\ref{eq.mod}), $\alpha\geq1$. Then there exists a $C>0$ (depending only on $\alpha, L, l, u$), such that
\begin{eqnarray*}
	\liminf_{n \rightarrow \infty} \inf_{\hat \sigma^2_n} 
	\sup_{\sigma^2 \in \mathcal{C}^b\left(\alpha,L \right)}
	\E\left(n^{\frac{\alpha}{2\alpha+1}}\left\|\hat \sigma^2-\sigma^2\right\|_2^2
	\right)\geq C.
\end{eqnarray*}
\end{thm}

\begin{proof}
Besides working with the observations $Y_{1,n},\ldots,Y_{n,n}$ directly, we consider the sufficient statistics
\begin{eqnarray*}
	\Delta Y
	&=&
	\left(
	Y_{1,n}, Y_{2,n}-Y_{1,n},\ldots,Y_{n,n}-Y_{n-1,n}\right),\\
	\Delta \tilde Y
	&=&
	\left(
	\tilde Y_{1,n}, \tilde Y_{2,n}- \tilde Y_{1,n},\ldots, \tilde Y_{n,n}-\tilde Y_{n-1,n}\right).
\end{eqnarray*}
Let $\Delta_i W:= W_{\left(i+1\right)/n}-W_{i/n}$ and $\Delta_i\epsilon_{i,n}=\epsilon_{i+1,n}-\epsilon_{i,n}.$ We obtain for $i=1,\ldots,n$
\begin{align}
	Y_{i,n}-Y_{i-1,n}
	&=
	\int_{\left(i-1\right)/n}^{i/n}
	\sigma(s)dW_s
	+
	\tau\Delta_{i-1}\epsilon_{i-1,n},
	\notag \\
	\tilde Y_{i,n}-\tilde Y_{i-1,n}
	&=
	\sigma\left(\frac in\right)
	\Delta_{i-1}W
	+
	\left(\Delta_{i-1}\sigma\right)
	W_{\left(i-1\right)/n}
	+\tau\Delta_{i-1}\epsilon_{i-1,n},
	\label{eq.expminimax}
\end{align}
where $\epsilon_{0,n}:=0$, $Y_{0,n}:=0$ and $\tilde Y_{0,n}:=0$. We may write $\Delta Y= X_1+X_2$ and $\Delta \tilde Y= X_1'+ X_2+R_1$, where $ X_1, X_1', X_2$ have components $\int_{\left(i-1\right)/n}^{i/n}\sigma(s)dW_s,$ $\sigma\left(i/n\right)
	\Delta_{i-1}W$  and $\tau\Delta_{i-1}\epsilon_{i-1,n}$, respectively and $\left( R_1\right)_i=\left(\Delta_{i-1}\sigma\right)
	W_{\left(i-1\right)/n}$.

\noindent
Let us first consider model (\ref{eq.mod}), $\alpha\geq1$. By the ease of brevity, we simply point out in a second step the differences to model (\ref{eq.mod2}). The idea is to prove the lower bound by a multiple testing argument. Explicitly, we apply Theorem 2.5 in Tsybakov \cite{tsyb}. The construction of hypothesis is similar to the one given in \cite{tsyb}, Section 2.6.1. We write $\sigma_{\min}, \sigma_{\max}$ for the lower and upper bound of $\sigma^2$, respectively, i.e. $\sigma^2\in \mathcal{C}^b\left(\alpha,L, \left[\sigma_{\min}, \sigma_{\max}\right] \right)$.  Without loss of generality, we may assume that $\sigma_{\min}=1$. Let
\begin{eqnarray}
	K: \mathbb{R}\rightarrow \mathbb{R}^+, \quad K\left(u\right)
	=a\exp\left(-\frac{1}{1-\left(2u\right)^2}\right)
	\mathbb{I}_{\left\{\left|2u\right|\leq 1\right\}}
	,
	\label{eq.Kerneldef}
\end{eqnarray}
where $a$ is such that $K\in \mathcal{C}\left(\alpha,1/2\right).$ Further for some $c>0,$ specified later on, let $m:=\left[ 2^{-1}cn^{1/\left(4\alpha+2\right)}+1\right]$, $h_n=\left(2m\right)^{-1}$, $t_k=h_n\left(k-1/2\right)+1/4$,
\begin{eqnarray*}
	\phi_k\left(t\right)
	:= Lh_n^\alpha K\left(\frac{t-t_k}{h_n}\right),
	\quad 
	k=1,\ldots, m, \quad t\in [0,1].
\end{eqnarray*}
Define $\Omega:=\left\{\omega=\left(\omega_1,\ldots,\omega_m\right),
	\omega_i\in \left\{0,1\right\}\right\}$ and consider the set \\ $\mathcal{E}:=\left\{\sigma^2_\omega\left(t\right) : \sigma^2_\omega\left(t\right) = 1 + \sum_{k=1}^m \omega_k\phi_k(t),  \omega  \in \Omega \right\}.$ Then it holds $\forall \omega, \omega' \in \Omega$
\begin{align}
	\left\|\sigma^2_\omega-\sigma^2_{\omega'}\right\|_2^2
	=
	\int_0^1\left(\sigma^2_\omega(t)-\sigma^2_{\omega'}(t)\right)^2 dt
	= L^2h_n^{2\alpha+1}\left\| K\right\|_2^2\rho\left(\omega,\omega'\right),
	\label{eq.disbound}
\end{align}
where $\rho\left(\omega,\omega'\right)=\sum_{k=1}^m \mathbb{I}_{\left\{\omega_k \neq \omega_k'\right\}}$ is the Hamming distance. By the Varshamov-Gilbert bound (cf. Tsybakov \cite{tsyb}) there exists for all $m\geq 8$ a subset $\left\{\omega_0,\ldots, \omega_M\right\}$ of $\Omega$ such that $\omega_0=\left(0,\ldots,0\right)$, $\rho\left(\omega_i,\omega_j\right)\geq m/8,$  $\forall \ 0\leq j<k\leq M$ and $M\geq 2^{m/8}$. Define the hypothesis $H_i$  for $i=0,\ldots,M$ by the probability measure $P_i$ induced by $\sigma^2_{i,n}:=\sigma^2_{\omega_i,n}$. By Theorem 2.5 in Tsybakov \cite{tsyb} the proof is finished once we have established
\begin{align}
	& (i) \quad  \sigma^2_{i,n}\in \mathcal{C}\left(\alpha,L\right), \ 1\leq \sigma^2_{i,n}\leq \sigma_{\max} \notag \\
	& (ii) \quad 
	\left\| \sigma^2_{i,n}-\sigma^2_{j,n}\right\|_2
	\geq 2s \geq cn^{-\alpha/\left(4\alpha+2\right)}, \quad i\neq j, \ c>0, \label{eq.lbconds} \\
	& (iii) \quad \frac 1M\sum_{j=1}^M d_K\left(P_j,P_0\right)
	\leq \kappa \log M, \quad j=1,\ldots,M, \quad \kappa<1/10. \notag
\end{align}
$(i)$ is obviously fulfilled for sufficiently large $n$. By (\ref{eq.disbound}) it follows for $m>8$, $\left\| \sigma^2_{i,n}-\sigma^2_{j,n}\right\|_2\geq 16^{-1}Lh_n^{\alpha}\left\|K\right\|_2$ and hence $(ii)$. $(iii)$ We apply Lemma \ref{lem.kllem} in combination with Lemma \ref{lem.posdefmaj}. Let $\Pi_k \in \mathbb{D}_n$, $k=1,\ldots,M$ with entries $\left(\Pi_k\right)_{i,j}=\sigma_{k,n}\left(i/n\right)\delta_{i,j}$. For the observation vector $\tilde Y=\left(\tilde Y_{1,n},\ldots,\tilde Y_{n,n}\right)$, we have $\tilde Y\sim \mathcal{N}\left(0,\Sigma'_k\right)$ under $H_k$, $k=0,\ldots,M$, where
\begin{eqnarray*}
	\Sigma'_0&=&\left(\frac{i\wedge j}n\right)_{i,j=1,\ldots,n}+\tau^2 I_n, \\
	\Sigma'_k&=&\Pi_k\left(\frac{i\wedge j}n\right)_{i,j=1,\ldots,n}
		\Pi_k +\tau^2 I_n, \quad k=1,\ldots,M.
\end{eqnarray*}
Because of $\sigma_{k,n}\geq 1$ it follows $\left|\sigma_{k,n}\left(x\right)-\sigma_{k,n}\left(y\right)\right| \leq \left|\sigma_{k,n}^2\left(x\right)-\sigma_{k,n}^2\left(y\right)\right|\leq L\left|x-y\right|,$ i.e. $\sigma \in \mathcal{C}\left(1,L\right).$ By Lemma \ref{lem.posdefmaj}, $0<\left(2+12L^2\right)^{-1}\Sigma_0'<\Sigma_k'.$ These inequalities remain valid under any invertible linear transformations of $\tilde Y$. Hence if we denote by $\Sigma_k$ the covariance of $\Delta \tilde Y$ under $H_k$ ($k=0,\ldots,M$) it follows $0< \left(2+12L^2\right)^{-1}\Sigma_0 < \Sigma_k$ and we may apply Lemma \ref{lem.kllem} with $C=\left(2+12L^2\right)^{-1}$. Hence
\begin{eqnarray*}
	d_K\left(P_k,P_0\right)
	&\leq&
	\frac{\left(2+12L^2 \right)^2}4
	\left\|\Sigma_0^{-1}\Sigma_k-I_n\right\|_F^2
	\quad k=1,\ldots,M.
\end{eqnarray*}
Let for $1\leq i,j\leq n$, the matrix $A$ be defined by
\begin{align}
	\left(A\right)_{i,j}
	:=
	\begin{cases}
	2 \quad \text{for} \quad  i=j \quad  \text{and} \quad  i>1  \\
	-1 \quad \text{for} \quad \left|i-j\right|=1 \\
	1 \quad \text{for} \quad i=j=1 \\
	0 \quad \text{else}
	\end{cases}
	\label{eq.Adef}
\end{align}
and let $\Gamma_k:=\Pi_k-I_n$, $k=1,\ldots,M$. Clearly, $\Gamma_k\leq Lh_n^\alpha \left\|K\right\|_\infty.$ We abbreviate the covariance of two column vectors $X$ and $Y$ as the matrix with covariances of $XY^t$. Then we have the explicit representations
\begin{eqnarray*}
	\Sigma_0&=& \frac 1n I_n+\tau^2 A,\\
	\Sigma_k&=& \Sigma_0+\frac 2n \Gamma_k+ \frac1n \Gamma_k^2+
	\Cov_{H_k}\left(X_1',R_1\right)
	+\Cov_{H_k}\left(R_1,X_1'\right)
	+\Cov_{H_k}\left(R_1\right),
\end{eqnarray*}
where the subscript $H_k$ means that these covariances are taken with respect to the probability measure induced by $H_k$. We remark that due to $\Sigma_0\geq I_n/n$ it holds $\lambda_1\left(\Sigma_0^{-1}\right)=\lambda_n^{-1}\left(\Sigma_0\right)\leq \lambda_n^{-1}\left(I_n/n\right)=n.$ This yields using Lemma \ref{lem.p}, (i)
\begin{align*}
	d_K\left(P_k,P_0\right)
	&\leq
	5
	\frac{\left(2+12L^2\right)^2}4
	\left(
	\frac 4{n^2} \left\|\Sigma_0^{-1}\Gamma_k\right\|_F^2
	+
	\frac 1{n^2}\left\|\Sigma_0^{-1}\Gamma_k^2\right\|_F^2
	+ 
	\left\|\Sigma_0^{-1}\Cov_{H_k}\left(X_1', R_1\right)\right\|_F^2
	\right. \notag \\ 
	&\quad \quad + \left. 
	\left\|\Sigma_0^{-1}\Cov_{H_k}\left(R_1,X_1'\right)\right\|_F^2
	+
	\left\|\Sigma_0^{-1}\Cov_{H_k}\left( R_1\right)\right\|_F^2
	\right) \notag \\
	&\leq
	5\left(1+6L^2\right)^2
	\left(
	\left(4L^2\left\|K\right\|_\infty^2+L^4\left\|K\right\|_\infty^4h_n^{2\alpha}\right)
	h_n^{2\alpha}n^{-2} \left\|\Sigma_0^{-1}\right\|_F^2
	\right. \notag  \\ 
	&\quad \quad 
	+ \left. 
	n^2\left\|\Cov_{H_k}\left(X_1',R_1\right)\right\|_F^2
	+
	n^2\left\|\Cov_{H_k}\left( R_1 , X_1'\right)\right\|_F^2
	+
	n^2\left\|\Cov_{H_k}\left( R_1\right)\right\|_F^2
	\right)
	.
\end{align*}
The remaining part of the proof is concerned with bounding these terms. We make use of the properties on Frobenius norms collected in Lemma \ref{lem.froblem} and obtain
\begin{eqnarray*}
	\left\|\Cov_{H_k}\left( R_1\right)\right\|_F^2
	&\leq&
	n^{-4}L^4\left\|\left(\frac{\left(i \wedge j\right)-1}n\right)_{i,j=1,\ldots,n}\right\|_F^2 
	\leq L^4n^{-2}.
\end{eqnarray*}
Let $E \in \mathbb{M}_n$ given by
\begin{eqnarray*}
	\left(E\right)_{i,j}:=
	\begin{cases}
	1 \quad \text{if} \quad j>i \\
	0 \quad \text{otherwise}
	\end{cases}
	.
\end{eqnarray*}
Also let $\Delta\Pi_k \in \mathbb{D}_n$, $k=1,\ldots,M$ with entries $\left(\Delta\Pi_k\right)_{i,j}= \left(\Delta_{i-1}\sigma_{k,n}\right)\delta_{i,j}$, where $\sigma_{k,n}(0)=0$. Then $\Cov_{H_k}\left(X_1', R_1\right)=n^{-1}\Pi_k E \left(\Delta\Pi_k\right)$ and for $n$ large enough
\begin{eqnarray*}
	\left\|\Cov_{H_k}\left(X_1', R_1\right)\right\|_F^2
	&\leq&
	4n^{-2}
	\Tr\left(E\left(\Delta\Pi_k\right)^2E^t\right)
	\leq
	4L^2n^{-4}\left\|E\right\|_F^2\leq 4L^2n^{-2}.
\end{eqnarray*}
With the same arguments we can bound  $\left\|\Sigma_0^{-1}\Cov_{H_k}\left( R_1, X_1'\right)\right\|_F^2$. Further we have for $j=1,\ldots,M$ using Lemma \ref{lem.eigenvalues}
\begin{align*}
	\left\|\Sigma_0^{-1}\right\|_F^2
	&\leq
	\sum_{i=1}^n\left(\frac 1n+\tau^2i^2/\left(4n^2\right)\right)^{-2}
	\leq C_\tau n^{5/2},
\end{align*}
for a constant $C_\tau$ only depending on $\tau.$ This gives
\begin{align*}
	&\frac1M \sum_{j=1}^M d_K\left(P_j,P_0\right) \\
	&\quad \leq
	5\left(1+6L^2\right)^2
	\left( \left(4L^2\left\|K\right\|_\infty^2+L^4\left\|K\right\|_\infty^4h_n^{2\alpha}\right)n^{1/2}h_n^{2\alpha} C_\tau+ L^4+8L^2
	\right) \\
	&\quad \leq C_{\tau,L,\left\|K\right\|_\infty}n^{1/2}h_n^{2\alpha}
	\leq C_{\tau,L,\left\|K\right\|_\infty} c^{-2\alpha-1}2m
	\leq
	\kappa \log M,
\end{align*}
where $C_{\tau,L,\left\|K\right\|_\infty}$ is independent of $n$ and the last inequality holds if $$c > \left(16 C_{\tau,L,\left\|K\right\|_\infty} \kappa^{-1}\right)^{1/\left(2\alpha+1\right)},$$ using the Varshamov-Gilbert bound.

The proof for model (\ref{eq.mod2}) and $\alpha>1/2$ is almost the same as for Theorem \ref{thm.lbl2}. So we only sketch it here. Note that Lemma \ref{lem.kllem} can be applied directly without use of Lemma \ref{lem.posdefmaj}. The construction of hypothesis is the same. Let $\Pi_k\in \mathbb{D}_n$, $k=1,\ldots,M$ defined by $\left(\Pi_k\right)_{i,i}:=\int_{\left(i-1\right)/n}^{i/n}\sigma^2_{k,n}\left(s\right)ds$ and $\Gamma_k:=\Pi_k-I_n$. With the notation as in Theorem \ref{thm.lbl2} the problem can be reduced to a testing problem where we have to test the hypothesis that a centered random vector has covariance matrix $\Sigma_0=I_n+A$ against the $M$ alternative covariance matrices $\Sigma_k=\Sigma_0+\Gamma_k, \ k=1,\ldots,M$. With $\max_{1\leq i\leq n}\max_{1\leq k\leq M}\left(\Gamma_k\right)_{i,i}=O\left(h_n^{\alpha}\right)$ and Lemma \ref{lem.kllem} the Theorem follows.
\end{proof}

The proof of the lower bound for estimation of $\sigma^2$ in model (\ref{eq.mod3}) differs from the previous one. Instead of using first order differences we transform the data by taking second order differences. The key step is to show that for constant $\sigma$ this is "close" to the model where we observe $Z_{i,n}=\sigma n^{-3/2} \eta_{i,n}+\tau \xi_{i,n}, i=1,\ldots,n, \ \eta_{i,n}\sim \mathcal{N}\left(0,1\right),$ i.i.d. Here, $\xi=\left(\xi_{1,n},\ldots, \xi_{n,n}\right)$ is a particular $\MA(2)-$process, independent of $\eta=\left(\eta_{1,n},\ldots,\eta_{n,n}\right).$

\begin{thm}
Assume model (\ref{eq.mod3}) and $\alpha>1/2$. Then there exists a $C>0$ (depending only on $\alpha, L, l, u$), such that
\begin{eqnarray*}
	\liminf_{n \rightarrow \infty} \inf_{\hat \sigma^2_n} 
	\sup_{\sigma^2 \in \mathcal{C}^b\left(\alpha,L \right)}
	\E\left(n^{\frac{\alpha}{4\alpha+2}}\left\|\hat \sigma^2-\sigma^2\right\|_2^2
	\right)\geq C.
\end{eqnarray*}
Furthermore, there exists a $\tilde C>0$ (depending on $\sigma_{\min},\sigma_{\max}$) such that for constant $\sigma$ and $0<\sigma_{\min}<\sigma_{\max}<\infty$
\begin{eqnarray*}
	\liminf_{n \rightarrow \infty} \inf_{\hat \sigma^2_n} 
	\sup_{\sigma^2 \in \left[\sigma_{\min},\sigma_{\max}\right]}
	\E\left(n^{1/4}\left(\hat \sigma^2-\sigma^2\right)^2
	\right)\geq \tilde C.
\end{eqnarray*}
\end{thm}

\begin{proof}
Except for changing the definition of $m:=\left[ 2^{-1}cn^{1/\left(4\alpha+2\right)}+1\right]$ to $m:=\left[ 2^{-1}cn^{1/\left(8\alpha+4\right)}+1\right]$ we construct the same hypothesis as in the proof of Theorem \ref{thm.lbl2}. In order to prove the first part of the statement it remains to show (iii) in (\ref{eq.lbconds}). We consider second order differences, i.e.
\begin{align*}
	Y_i^*:=\Delta^2_{i-1} \bar Y_i, \ \bar Y_0:=0, \ i=2,3,\ldots, 
	\ Y_1^*:= \sqrt{2}\bar Y_1,
\end{align*}
where $\Delta^2_i=\Delta_i \circ \Delta_i.$ Note that for $i\geq 2$
\begin{align*}
	Y_i^*=\int_{\left(i-1\right)/n}^{i/n} \left(\frac in-s\right)\sigma\left(s\right) dW_s
	+\int_{\left(i-2\right)/n}^{\left(i-1\right)/n} \left(s-\frac {i-2}n\right)\sigma\left(s\right) dW_s
	+\tau \Delta^2_i \epsilon_i
\end{align*}
Let $Y^*:=\left(Y^*_{1},\ldots,Y^*_n\right).$ Obviously, this is equivalent to observing $\bar Y.$ The covariance of $Y^*$ under hypothesis $H_k$ is denoted by $\Sigma_k, k=0,\ldots,M.$ Since by construction $\sigma_{0,n}^2\leq \sigma_{k,n}^2,$ it follows from elementary computations that $\Sigma_k-\Sigma_0$ is the covariance of $Y^*$ under $\sigma=\left(\sigma_{k,n}^2-\sigma_{0,n}^2\right)^{1/2}$ and $\tau=0.$ From this we conclude that 
\begin{align}
	\Sigma_k-\Sigma_0\geq 0,
	\label{eq.sks}
\end{align}
i.e. $\Sigma_0\leq \Sigma_k.$ Note that $\Var\left(X+Y\right)\leq 2\Var X+2\Var Y$ (in the sense of Loewner ordering, see also Lemma \ref{lem.p}, (iii)). This yields together with (\ref{eq.sks})
\begin{align*}
	\Sigma_k-\Sigma_0\leq \Gamma,
\end{align*}
where $\Gamma$ is diagonal with entries
\begin{align*}
	\left(\Gamma\right)_{i,j}=
	\frac{4Lh_n^\alpha \left\|K\right\|_\infty}{3n^3}\delta_{i,j},
\end{align*}
and $\delta_{i,j}$ denotes the Kronecker delta. Let $P_k$ denote the probability measure of $Y^*_k$ under $H_k.$ Due to $\Sigma_0\leq \Sigma_k$ we may apply Lemma \ref{lem.kllem} and obtain
\begin{align}
	d_K\left(P_k,P_0\right)
	\leq \frac 14\left\|\Sigma_0^{-1/2}\left(\Sigma_k-\Sigma_0\right)\Sigma_0^{-1/2}\right\|_F^2
	\leq \frac{16L^2h_n^{2\alpha} \left\|K\right\|_\infty^2}{9n^6}
	\left\|\Sigma_0^{-1}\right\|_F^2.
	\label{eq.klbound}
\end{align}
Direct computations give
\begin{align*}
	\Sigma_0
	=\frac 1{n^3}I_n-\frac 1{6n^3}A+\frac{\sqrt{2}-1}{6n^3}
	V_1
	+\tau^2\left(A^2+V_2\right),
\end{align*}
where $A$ is as defined in (\ref{eq.Adef}),
\begin{align*}
	\left(V_{1}\right)_{i,j}:=
	\begin{cases}
	1 &\text{if} \ i=1, j=2 \ \text{or} \  i=2, j=1,\\
	0 &\text{otherwise}.
	\end{cases}
\end{align*}
and $V_2$ is symmetric and $\left(V_2\right)_{i,j}\neq 0$ only if $i,j\leq 3.$ Obviously, the smallest eigenvalue of $V_1$ is $-1.$ Hence we can estimate by Lemma \ref{lem.eigenvalues},
\begin{align*}
	\Sigma_0\geq \frac1{6n^3}I_n+\tau^2\left(A^2+V_2\right).
\end{align*}
Since $V_2$ has only non-zero entries in the first three rows and columns, it has only three non-zero eigenvalues. By standard bounds on eigenvalues (see Lemma \ref{lem.p}, (ii)) this allows to estimate for $i\geq 3$
\begin{align*}
	\lambda_{n-i}\left(A^2+V_2\right)
	\geq 
	\lambda_{n-i+3}\left(A^2\right)+\lambda_{n-3}\left(V_2\right)
	= \lambda_{n-i+3}^2\left(A\right).
\end{align*}
Let $r_n:=\left[n^{1/4}\right].$ Then for sufficiently large $n$ by Lemma \ref{lem.eigenvalues}
\begin{align*}
	\left\|\Sigma_0^{-1}\right\|_F^2
	&=\sum_{i=1}^n\lambda_i^2\left(\Sigma_0^{-1}\right)
	\leq \sum_{i=1}^{r_n} \lambda_{n-i+1}^{-2}\left(\frac 1{6n^3}I_n\right)
	+\tau^4 \sum_{i=r_n+1}^n\lambda_{n-i+1}^{-2}
	\left(A^2+V_2\right) \\
	&\leq 
	36n^{25/4}+\tau^4\sum_{i=r_n+1}^n \lambda_{n-i+4}^{-4}\left(A\right)
	\leq
	36n^{25/4}+4^4\tau^4\sum_{i=r_n+1}^n \frac{n^8}{\left(i-3\right)^8} \\
	&\leq 36n^{25/4}+2^{24}n^{25/4}\frac 1{n^{1/4}}
	\sum_{i=r_n+1}^n \frac 1{\left(i/n^{1/4}\right)^8}.
\end{align*}
Because the last part converges as a Riemann sum to a finite integral, we may find a constant $C_\tau$ depending only on $\tau$ such that $\left\|\Sigma_0^{-1}\right\|_F^2 \leq C_\tau n^{25/4}.$ This gives with (\ref{eq.klbound}), 
\begin{align*}
	\frac 1M\sum_{k=1}^M d_K\left(P_k,P_0\right)
	\leq 
	\frac{16}9 C_\tau L^2\left\|K\right\|_{\infty}^2
	h_n^{2\alpha}n^{1/4}
	\leq \frac{2^8}9 C_\tau L^2\left\|K\right\|_{\infty}^2
	c^{-2\alpha-1}m \leq \kappa \log M
\end{align*}
whenever $c\geq \left(2^8/9 C_\tau L^2\left\|K\right\|_{\infty}^2 \kappa^{-1} \right)^{1/\left(2\alpha+1\right)}.$

\bigskip

\noindent
In order to prove the second statement of the theorem we consider two hypothesis $\sigma_0^2:=\sigma_{\min}$ and $\sigma^2_{1}:= \sigma_{\min}+cn^{-1/8}$ and apply Theorem 2.5 in Tsybakov \cite{tsyb} for $M=2.$ Using the bounds from the first part, the remaining part of the proof is straightforward and thus omitted.
\end{proof}

\begin{appendices}

\section{Technical tools}\label{app}

\begin{lem}
\label{lem.eigenvalues}
Let $A, Q^{-1}$ be as in (\ref{eq.Adef}) and (\ref{eq.Qminus1}), respectively. Then 
\begin{align*}
	\lambda_{n-i+1}\left(A\right)=\lambda_{n-i+1}\left(Q^{-1}\right)
	= 4\sin^2\left(\frac{\left(2i-1\right)\pi}{4n+2}\right)
	\geq \frac{i^2}{4n^2}, \quad i=1,\ldots,n
\end{align*}
\end{lem}

\begin{proof}
The $i$-th eigenvector $v_i$ of $Q^{-1}$ is given by $v_i=\left(\sin\left(x_i\right), \sin\left(2x_i\right),\ldots, \sin\left(nx_i\right)\right),$ where $x_i:= \left(2i-1\right)\pi/\left(2n+1\right), i=1,\ldots,n.$ The corresponding eigenvalues are
\begin{align*}
	\lambda_{n-i+1}\left(Q^{-1}\right)
	= 4\sin^2\left(\frac{\left(2i-1\right)\pi}{4n+2}\right).
\end{align*}
If $v_i=\left(v_{i,1},\ldots,v_{i,n}\right)$ is an eigenvector of $Q^{-1}$ so is $\tilde v_i =\left(v_{i,n},\ldots,v_{1,1}\right)$ an eigenvector of $A$ with the same eigenvalue. Using $x\pi/2\leq \sin\left(x\pi\right)$ whenever $x\in \left[0,1/2\right],$ we obtain
\begin{align*}
	4\sin^2\left(x_i\pi/2\right)
	\geq x_i^2\pi^2/4\geq \frac{i^2}{4\left(2n+1\right)^2}\pi^2 \geq \frac{i^2}{4n^2}.
\end{align*}
\end{proof}

\begin{lem}
\label{lem.posdefmaj}
Let $\sigma\geq1$ be a H\"older $\mathcal{C}(1,L)$ function and let $\Sigma \in \mathbb{D}_n$ be a diagonal matrix with $\Sigma_{i,i}=\sigma\left(i/n\right)$. Further introduce $Q:= \left(i\wedge j\right)_{i,j=1,\ldots,n}$. Then 
\begin{eqnarray*}
	\left(2+12L^2\right)^{-1}Q
	\leq
	\Sigma Q \Sigma
\end{eqnarray*}
in the sense of partial Loewner ordering of symmetric matrices.
\end{lem}


\begin{proof}
Obviously
\begin{align}
	Q^{-1}_{i,j}
	=
	\begin{cases}
	2 \quad &\text{for} \quad  i=j \quad  \text{and} \quad  i<n,  \\
	-1 \quad &\text{for} \quad \left|i-j\right|=1,  \\
	1 \quad &\text{for} \quad i=j=n,  \\
	0 \quad &\text{else}.
	\end{cases}
	\label{eq.Qminus1}
\end{align}
Let $O\in \mathbb{M}_n$ given by
\begin{align*}
	O_{i,j}
	=
	\begin{cases}
	1 \quad &\text{for} \quad  i=j, \\
	-1 \quad &\text{for} \quad i=j-1, \\
	0 \quad &\text{else}.
	\end{cases}
\end{align*}
Note that $Q^{-1}=OO^t$. We have with
\begin{align*}
	\left(\tilde \Sigma\right)_{i,j}
	:=
	\begin{cases}
	\left(\Delta_i \sigma^{-1}\right)^2 \quad
	& \text{for} \quad i=j-1, \\
	\left(\Delta_j \sigma^{-1}\right)^2 \quad
	& \text{for} \quad i=j+1, \\
	0 \quad & \text{else}
	\end{cases}
\end{align*}
that
\begin{align}
	&\Sigma^{-1}Q^{-1} \Sigma^{-1}
	= \frac 12 \Sigma^{-2} Q^{-1}
	+ \frac 12 Q^{-1}\Sigma^{-2}
	+
	\frac 12 \tilde \Sigma \notag \\
	&\quad \leq
	Q^{-1} 
	+ \frac 12 \left(\Sigma^{-2}O-O\Sigma^{-2}\right)O^t
	+ \frac 12 O \left(O^t\Sigma^{-2}-\Sigma^{-2}O^t\right)
	+\frac 12 \tilde \Sigma.
	\label{eq.sigmaineq}
\end{align}
As can be seen by direct calculations
\begin{align}
 	&\left(\Sigma^{-2}O-O\Sigma^{-2}\right)O^t
	+ O \left(O^t\Sigma^{-2}-\Sigma^{-2}O^t\right) \notag \\
	& \quad = -
	\left(
	\begin{array}{cccc}
	2\Delta_1\sigma^{-2} & -\Delta_1\sigma^{-2} & & \\
	-\Delta_1\sigma^{-2} & \ddots & \ddots & \\
	& \ddots &	2\Delta_{n-1}\sigma^{-2} & -\Delta_{n-1}\sigma^{-2} \\
	& & -\Delta_{n-1}\sigma^{-2} & 0
	\end{array}
	\right).
	\label{eq.commub}
\end{align}
Define $M \in \mathbb{M}_n$ by
\begin{align*}
	M_{i,j}
	=
	\begin{cases}
	1 \quad &\text{for} \quad  i=j, \\
	\Delta_i \sigma^{-2}-1 \quad &\text{for} \quad i=j-1, \\
	0 \quad &\text{else}.
	\end{cases}
\end{align*}
Due to $MM^t\geq 0$ and (\ref{eq.commub}) it follows
\begin{align*}
 	&\left(\Sigma^{-2}O-O\Sigma^{-2}\right)O^t
	+ O \left(O^t\Sigma^{-2}-\Sigma^{-2}O^t\right) \\
	& \quad = Q^{-1}+
	\left(
	\begin{array}{cccc}
	\left(\Delta_1 \sigma^{-2}\right)^2 & & & \\
	& \ddots & & \\
	& & \left(\Delta_{n-1} \sigma^{-2}\right)^2 &  \\
	& & & 0
	\end{array}
	\right)
	\leq Q^{-1}+\frac{4L^2}{n^2}I_n,
\end{align*}
where we used in the last step that
\begin{align*}
	\left|\frac 1{\sigma^2\left(x\right)}-\frac 1{\sigma^2\left(y\right)}
	\right|
	= \left|\sigma\left(x\right)-\sigma\left(y\right)\right|
	\left|\frac{1}{\sigma^2\left(x\right)\sigma\left(y\right)}
	+\frac{1}{\sigma\left(x\right)\sigma^2\left(y\right)}\right|
	\leq 2L\left|x-y\right|,
\end{align*}
i.e. $\sigma\in \mathcal{C}\left(1,L\right)$ and $\sigma\geq 1$ implies $\sigma^{-2}\in \mathcal{C}\left(1,2L\right)$.

Next we bound the eigenvalues of the symmetric matrix $\tilde \Sigma$ by showing that the corresponding characteristic polynomial $\chi_{\tilde \Sigma}(t)$ does not have any zero in $[2\omega,\infty)$, where $\omega =\max_{i,j}\left|\left(\tilde \Sigma\right)_{i,j} \right|$. In order to see this introduce the notation $s(i):= \left(\tilde \Sigma\right)_{i,i+1}/\omega$ and note for $t\geq2\omega$
\begin{align*}
	&\left(
	\begin{array}{cccc}
	t\omega^{-1} & -s\left(1\right)& & \\
	-s\left(1\right) & \ddots & \ddots & \\
	& \ddots &	\ddots & -s\left(n-1\right) \\
	& & -s\left(n-1\right) & t\omega^{-1}
	\end{array}
	\right) \\
 	& \ \geq
	\left(
	\begin{array}{cccc}
	1+s\left(1\right)^2 & -s\left(1\right) & & \\
	-s\left(1\right) & \ddots & \ddots & \\
	& \ddots &	1+s\left(n-1\right)^2 & -s\left(n-1\right) \\
	& & -s\left(n-1\right)^2 & 1
	\end{array}
	\right) \\
	& \ =
	\left(
	\begin{array}{cccc}
	1& -s\left(1\right)& & \\
	& \ddots & \ddots & \\
	& & \ddots & -s\left(n-1\right) \\
	& & & 1
	\end{array}
	\right)
	\left(
	\begin{array}{cccc}
	1& -s\left(1\right)& & \\
	& \ddots & \ddots & \\
	& & \ddots & -s\left(n-1\right) \\
	& & & 1
	\end{array}
	\right)^t >0
\end{align*}
and therefore $\chi_{\tilde \Sigma}(t)>0$ for $t\geq 2\omega$. Because of $w\leq L^2/n^2$ this shows that 
\begin{eqnarray}
	\Sigma^{-1}Q^{-1} \Sigma^{-1}
	\leq
	3/2Q^{-1}+\frac{3L^2}{n^2} I_n.
	\label{eq.lslineq}
\end{eqnarray}
From Lemma \ref{lem.eigenvalues} follows $1/(4n^2)\leq \lambda_1^{-1}\left(Q\right)=\lambda_n\left(Q^{-1}\right)$ and hence
\begin{eqnarray*}
	\frac{L^2}{n^2}I_n
	\leq 4L^2\lambda_n\left(Q^{-1}\right) I_n
	\leq
	4L^2 Q^{-1}.
\end{eqnarray*}
This gives with (\ref{eq.lslineq}) finally
\begin{eqnarray*}
	\Sigma^{-1}Q^{-1} \Sigma^{-1}
	\leq
	2Q^{-1} +
	12L^2 Q^{-1}.
\end{eqnarray*}
\end{proof}

In the next lemma we collect some important facts about positive semidefinite and Hermitian matrices.
\begin{lem}
\label{lem.p}
\begin{itemize}
\item[(i)]
Let $A, B  \in \mathbb{M}_{n}$ are positive semidefinite matrices. Denote by $\lambda_1(A)$ the largest eigenvalue of $A$. Then $\Tr(AB)\leq \lambda_1(A)\Tr(B)$.
\item[(ii)] Let $A, B \in \mathbb{M}_{n}$ are Hermitian. Then 
\begin{eqnarray*}
	\lambda_{n-r-s}\left(A+B\right)
	&\geq& \lambda_{n-r}\left(A\right)+\lambda_{n-s}\left(B\right).
\end{eqnarray*}
\item[(iii)] Let $A, B$ are matrices of the same size. Then $$A^tB+B^tA\leq A^tA+B^tB.$$ 
\end{itemize}
\end{lem}

\noindent
\begin{lem}[Frobenius norm]
\label{lem.froblem}
\begin{itemize} Let $A \in \mathbb{M}_{n}$. Then
\item[(i)]
\begin{eqnarray*}
	\left\|A\right\|_F^2:=\Tr\left(AA^t\right)
	=\sum_{i=1}^{n} \lambda_i\left(AA^t\right)
	=\sum_{i,j=1}^{n} a^2_{i,j}
\end{eqnarray*}
and whenever $A=A^t$ also $\left\|A\right\|_F^2=\sum_{i=1}^{n} \lambda_i^2\left(A\right)$.
\item[(ii)]
It holds
\begin{eqnarray*}
	4\Tr\left(A^2\right)\leq \left\|A+A^t\right\|_F^2\leq 4\left\|A\right\|_F^2.
\end{eqnarray*}
\item[(iii)]
Let $A$, $B$ be positive semidefinite matrices of the same size and $0\leq A\leq B$. Further let $X$ be another matrix of the same size. Then
\begin{eqnarray*}
	\left\|X^tAX\right\|_F
	\leq
	\left\|X^tBX\right\|_F.
\end{eqnarray*}
\end{itemize}
\end{lem}

\begin{proof}
(i) and (ii) is well known and omitted. (iii) By assumption it holds $0\leq X^t A X \leq X^t B X$. Hence $\lambda_i^2\left(X^t A X\right)\leq \lambda_i^2 \left( X^t B X\right)$ and the result follows.
\end{proof}

\end{appendices}

%


\section*{Acknowledgments}
The first author is indepted to Larry D. Brown for his generous and warm hospitality during several visits to Philadelphia. Collaboration and discussion with Larry were always a great source of inspiration and had significant impact on his work.

\bibliographystyle{plain}       
\bibliography{refsPart1}

\end{document}